\documentclass[12pt,leqno]{amsart}

\usepackage{amsmath,amssymb,amsthm}
\newtheorem{theorem}{Theorem}[section]
\newtheorem{lemma}[theorem]{Lemma}

\newtheorem{proposition}[theorem]{Proposition}
\newtheorem{corollary}[theorem]{Corollary}
\theoremstyle{definition}

\theoremstyle{remark}
\newtheorem{remark}[theorem]{Remark}
\numberwithin{equation}{section}
\def\fnote#1{\footnote}

\def\natu{{\mathbb N}}

\def\ignora#1{}
\def\n3#1{\left\vert  \! \left\vert \! \left\vert \, #1 \, \right\vert \!
  \right\vert \! \right\vert }

\begin{document}

\title{ Weak-star point of continuity property and Schauder bases }
%\author{}
%\address{Universidad de Granada, Facultad de Ciencias.
%Departamento de Matem\'{a}tica Aplicada, 18071-Granada (Spain)}

\author{Gin{\'e}s L{\'o}pez-P{\'e}rez and Jos{\'e} A. Soler Arias}
\address{Universidad de Granada, Facultad de Ciencias.
Departamento de An\'{a}lisis Matem\'{a}tico, 18071-Granada
(Spain)} \email{glopezp@ugr.es, jasoler@ugr.es}

\thanks{Partially supported by MEC (Spain) Grant MTM2006-04837 and Junta de Andaluc\'{\i}a Grants FQM-185 and Proyecto
de Excelencia P06-FQM-01438.} \subjclass{46B20, 46B22. Key words:
 Point of continuity property, trees, boundedly complete sequences}
 \maketitle \markboth{G. L\'{o}pez and Jos{\'e} A. Soler   }{
 Weak-star PCP and Schauder bases }

\begin{abstract}
  We characterize  the weak-star point of
continuity property for subspaces of dual spaces with separable
predual and we deduce that the weak-star point of continuity
property is determined by subspaces with a Schauder basis in the
natural setting of dual spaces of separable Banach spaces. As a
consequence of the above characterization we get that a dual space
satisfies the Radon-Nikodym property if, and only if, every
seminormalized topologically weak-star null tree has a boundedly
complete branch, which improves  some results in \cite{DF}
obtained for the separable case. Also, as a consequence of the
above characterization, the following result obtained in \cite{R1}
is deduced: {\it every seminormalized basic sequence in a Banach
space with the point of continuity property has a boundedly
complete subsequence.}

\end{abstract}

\section{Introduction}
\par
\bigskip

We recall (see \cite{bou} for background) that a bounded subset
$C$ of a Banach space satisfies the Radon-Nikodym property (RNP)
if every subset of $C$ is dentable, that is, every subset of $C$
has slices of diameter arbitrarily small. A Banach space is said
to verify the RNP whenever its closed unit ball satisfies the RNP.
It is well known that separable dual spaces have RNP and spaces
with RNP contain many subspaces which are themselves separable
dual spaces. (Note that containing many separable dual subspaces
is equivalent to containing many boundedly complete basic
sequences). As RNP is separably determined, that is, a Banach
space $X$ has RNP whenever every separable subspace of $X$ has
RNP, it seems natural looking for a sequential characterization of
RNP in terms of boundedly complete basic sequences. In \cite{DF}
is proved that the space $B_{\infty}$ (which fails to have RNP)
still has the property: any $w$-null normalized sequence has a
boundedly complete basic subsequence. However, it has been proved
in \cite{DF} that the dual space of a separable Banach space $X$
has RNP if, and only if, every weak-star null tree in the unit
sphere of $X^*$ has some boundedly complete basic branch. It seems
then natural looking for a characterization of RNP for general
dual Banach spaces in terms of boundedly complete basic sequences,
extending the result in \cite{DF} proved for dual of separable
Banach spaces. For this, we introduce the concept of topologically
weak-star null tree, which is a weaker condition than the
weak-star null tree condition, and we characterize in terms of
trees the RNP for weak-star compact subsets of general dual Banach
spaces in proposition \ref{r1}. As a consequence, we get in
theorem \ref{r2} that  a dual Banach space $X$ has RNP if, and
only if, every seminormalized and topologically weak-star null
tree in the unit sphere of $X$ has some boundedly complete branch,
which has as an immediate corollary the aforementioned result in
\cite{DF}.

We recall that a closed and bounded subset of a Banach space $X$
satisfies the point of continuity property (PCP) if every closed
subset of $C$ has some point of weak continuity, that is, the weak
and the norm topologies agree at this point. Also, when $X$ is a
dual space, $C$ is said to satisfy the weak-star point of
continuity property ($w^*$-PCP) if every closed subset of $C$ has
some point of weak-$*$ continuity, equivalently every nonempty
subset of $C$ has relatively $w^*$-open subsets with diameter
arbitrarily small. $X$ has PCP (resp. $w^*$-PCP when $X$ is a dual
space) if $B_X$, the closed unit ball of $X$, has PCP (resp.
$w^*$-PCP). Also, a subspace $X$ of a dual space $Y^*$ is said to
verify the $w^*$-PCP if $B_X$, as a subset of $Y^*$, has the
$w^*$-PCP.   It is well known that RNP implies PCP, being false
the converse, and it is clear that $w^*$-PCP implies PCP.
Moreover, RNP and $w^*$-PCP are equivalent for convex
$w^*$-compact sets in a dual space, see theorem 4.2.13 in
\cite{bou}. We will use this last fact freely in the future. We
refer to \cite{R2} for background about PCP and $w^*$-PCP. It is a
well known open problem \cite{B} if PCP (resp. RNP) is determined
by subspaces with a Schauder basis. Our goal is characterize
$w^*$-PCP for closed and bounded subsets of dual spaces of
separable Banach spaces and conclude in theorem \ref{fin} that, in
fact, $w^*$-PCP is determined by subspaces with a Schauder basis
in the  natural setting of subspaces of dual spaces with a
separable predual. As an easy consequence we also deduce from the
above characterization of $w^*$-PCP that every seminormalized
basic sequence in a Banach space with PCP has a boundedly complete
basic subsequence. This last result was obtained in \cite{R1}.

We begin with some notation and preliminaries. Let $X$ be a Banach
space and let $B_X$, respectively $S_X$, be the closed unit ball,
respectively sphere, of $X$. The weak-star topology in $X$, when
it is a dual space, will be denoted by $w^*$. If $A$ is a subset
in $X$, $\overline{A}^{w^*}$ stands for the weak-star closure of
$A$ in $X$.  Given $\{e_n\}$ a basic sequence in $X$, $\{e_n\}$ is
said to be {\it semi-normalized} if $0<\inf_n\Vert e_n\Vert\leq
\sup_n\Vert e_n\Vert <\infty$ and the closed linear span of
$\{e_n\}$ is denoted by $[e_n]$. $\{e_n\}$ is called {\it
boundedly complete} provided whenever scalars $\{\lambda_i\}$
satisfy $\sup_n\Vert\sum_{i=1}^n\lambda_ie_i\Vert<\infty$, then
$\sum_n\lambda_ne_n$ converges. $\{e_n\}$ is called {\it
shrinking} if $[e_n]^*=[e_n^*]$, where $\{e_n^*\}$ denotes the
sequence of biorthogonal functionals associated to $\{e_n\}$.

A boundedly complete basic sequence $\{e_n\}$ in a Banach space
$X$ spans a dual space. In fact, $[e_n^*]^*=[e_n]$, where
$\{e_n^*\}$ denotes the sequence of biorthogonal functionals in
the dual space $X^*$ \cite {LZ}. Following the notation in
\cite{S}, it is said that a sequence $\{e_n\}$ in a Banach space
is {\it type P} if the set $\{\sum_{k=1}^ne_k:n\in \natu\}$ is
bounded. Observe, from the definitions, that  type P
seminormalized basic sequences fail to be always boundedly
complete basic sequences.

A   sequence $\{x_n\}$ in a Banach space $X$ is said to be {\it
strongly summing} if whenever $\{\lambda_n\}$ is a sequence of
scalars with $\sup_n\Vert\sum_{k=1}^n\lambda_kx_k\Vert<\infty$ one
has that the series of scalars $\sum_n\lambda_n$ converges. The
remarkable $c_0$-theorem \cite{R3} assures that every weak-Cauchy
and not weakly convergent sequence in a Banach space not
containing subspaces isomorphic to $c_0$ has a strongly summing
basic subsequence.

$\natu^{<\omega}$ stands for the set of all ordered finite
sequences of natural numbers joint to the empty sequence denoted
by $\emptyset$. We consider the natural order in
$\natu^{<\omega}$, that is, given
$\alpha=(\alpha_1,\ldots,\alpha_p),\
\beta=(\beta_1,\ldots,\beta_q)\in \natu^{<\omega}$, one has
$\alpha\leq \beta$ if $p\leq q$ and $\alpha_i=\beta_i$ $\forall
1\leq i\leq p$. If $\alpha=(\alpha_1,\ldots,\alpha_p)\in
\natu^{<\omega}$ we do $\alpha -=(\alpha_1,\ldots,\alpha_{p-1})$.
Also $\vert \alpha\vert$ denotes the {\it length} of sequence
$\alpha$, and $\emptyset$ is the minimum of $\natu^{<\omega}$ with
this partial order. A {\it tree} in a Banach space $X$ is a family
$\{x_A\}_{A\in \natu^{<\omega}}$ of vectors in $X$ indexed on
$\natu^{<\omega}$. The tree will be said {\it seminormalized} if
$0<\inf_A\Vert x_A\Vert\leq \sup_A\Vert x_A\Vert <\infty$. We will
say that the tree $\{x_A\}_{A\in \natu^{<\omega}}$ is {\it
$w^*$-null}, when $X$ is a dual space, if the sequence
$\{x_{(A,n)}\}_n$ is $w^*$-null for every $A\in\natu^{<\omega}$.
The tree $\{x_A\}_{A\in \natu^{<\omega}}$ is {\it topologically
$w^*$-null} if $0\in \overline{\{x_{(A,n)}:n\in \natu\}}^{w^*}$
for every $A\in\natu^{<\omega}$. A sequence $\{x_{A_n}\}_{n\geq
0}$ is called a {\it branch} if $\{A_n\}$ is a maximal totally
ordered subset of $\natu^{<\omega}$, that is, there exists a
sequence $\{\alpha_n\}$ of natural numbers such that
$A_n=(\alpha_1,\ldots,\alpha_n)$ for every $n\in \natu$ and
$A_0=\emptyset$. Given a tree $\{x_A\}_{A\in\natu^{<\omega}}$ in a
Banach space, a {\it full subtree} is a new tree
$\{y_A\}_{A\in\natu^{<\omega}}$ defined by
$y_{\emptyset}=x_{\emptyset}$ and $y_{(A,n)}=x_{(A,\sigma_A(n))}$
for every $A\in \natu^{<\omega}$ and for every $n\in \natu$, where
for every $A\in \natu^{<\omega}$, $\sigma_A$ is a strictly
increasing map, equivalently when every branch of $\{y_A\}$ is
also a branch of $\{x_A\}$. The tree $\{x_A\}_{A\in
\natu^{<\omega}}$ is said to be {\it uniformly type P} if every
branch of the  tree is type P and the partial sums of every branch
are uniformly bounded. The tree $\{x_A\}_{A\in \natu^{<\omega}}$
is said to be {\it basic} if the countable set $\{x_A:\ A\in
\natu^{\omega}\}$ is a basic sequence for some rearrangement.

Whenever $\{x_n\}$ is a sequence in a Banach space $X$, we will
see this sequence also like a tree doing $x_A=x_{\max(A)}$ for
every $A\in\natu^{<\omega}$. Furthermore the branches of this tree
are the subsequences of the sequence $\{x_n\}$.

Finally, we recall that a {\it boundedly complete skipped blocking
finite dimensional decomposition} (BCSBFDD) in a separable Banach
space $X$ is a sequence $\{F_j\}$ of finite dimensional subspaces
in $X$ such that:\begin{enumerate} \item $X=[F_j:j\in
\natu]$.\item $F_k\cap[F_j:j\neq k]=\{0\}$ for every $k\in
\natu$.\item For every sequence $\{n_j\}$ of non-negative integers
with $n_j+1<n_{j+1}$ for all $j\in \natu$ and for every $f\in
[F_{(n_j,n_{j+1})}:j\in \natu]$ there exists a unique sequence
$\{f_j\}$ with $f_j\in F_{(n_j,n_{j+1})}$ for all $j\in \natu$
such that $f=\sum_{j=1}^{\infty}f_j$.\item Whenever $f_j\in
F_{(n_j,n_{j+1})}$ for all $j\in \natu$ and $\sup_n\Vert
\sum_{j=1}^{n}f_j\Vert<\infty$ then $\sum_{j=1}^{\infty}f_j$
converges.\end{enumerate} If $X$ is a subspace of $Y^*$ for some
$Y$, a BCSBFDD $\{F_j\}$ in $X$ will be called $w^*$-continuous if
$F_i\cap \overline{[F_j:j\neq i]}^{w^*}= \{0\}$ for every $i$.
Here, $[A]$ denotes the closed linear span in $X$ of the set $A$
and, for some nonempty interval of non-negative integers $I$, we
denote the linear span of the $F_j'$s for $j\in I$ by $F_{I}$.

If $\{F_j\}$ is a BCSBFDD in a separable Banach space $X$ and
$\{x_j\}$ is a sequence in $X$ such that $x_j\in
F_{(n_j,n_{j+1})}$ for some sequence $\{n_j\}$ of non-negative
integers with $n_j+1<n_{j+1}$ for all $j\in \natu$, we say that
$\{x_j\}$ is a {\it skipped block sequence} of $\{F_n\}$. It is
standard to prove that there is a positive constant $K$ such that
every skipped block sequence $\{x_j\}$ of $\{F_n\}$ with $x_j\neq
0$ for every $j$ is a boundedly complete basic sequence with
constant at most $K$.

From \cite{GM}, we know that the family of separable Banach spaces
with PCP is exactly the family of separable Banach spaces with a
BCSBFDD.

\section{Main results}
\par
\bigskip

We begin with a characterization of RNP for $w^*$-compact of
general dual spaces. This result can be seen like a $w^*$-version
of results in \cite{LS}.

\begin{proposition}\label{r1} Let $X$ be a Banach space and let $K$ be a weak-star compact and convex subset of $X^*$.
 Then the following
assertions are equivalent:\begin{enumerate} \item[i)] $K$ fails
RNP. \item[ii)] There is a seminormalized topologically weak-star
null tree $\{x_A\}_{A\in\natu^{<\omega}}$ in $X^*$ such that
$\{\sum_{C\leq A}x_C:A\in\natu^{<\omega}\}\subset
K$.\end{enumerate}\end{proposition}

\begin{proof} i)$\Rightarrow$ ii) Assume that $K$ fails RNP. Then,
from theorem 2.3.6 in \cite{bou}
 there is $D$ a non-dentable and countable subset of $K$. Now
$\overline{co}^{w^*}(D)$ is a weak-star compact and weak-star
separable subset of $K$ failing $w^*$-PCP. So there is $B$ a
relatively weak-star separable subset of $\overline{co}^{w^*}(D)$
and $\delta>0$ such that every relatively weak-star open subset of
$B$ has diameter greater than $2\delta$. So $b\in
\overline{B\setminus B(b,\delta)}^{w^*}$ for every $b\in B$, where
$B(b,\delta)$ stands for the open ball with center $b$ and radius
$\delta$. Note then that, since $B$ is relatively weak-star
separable, for every $b\in B$ there is a countable subset $C_b\in
B\setminus B(b,\delta)$ such that $b\in \overline{C_b}^{w^*}$.

First, we construct a tree $\{y_A\}_{A\in \natu^{<\omega}}$ in $B$
satisfying:\begin{enumerate}\item[a)]$y_{A}\in\overline{B\setminus
B(y_A,\delta)}^{^*}w$ for every $A\in\natu^{<\omega}$.\item[b)]
$\Vert y_A-y_{(A,i)}\Vert>\delta$ for every
$A\in\natu^{<\omega}$.\item[c)] $y_A\in
\overline{\{y_{(A,i)}:i\in\natu\}}^{w^*}$ for every
$A\in\natu^{<\omega}$.  \end{enumerate} Pick $y_{\emptyset}\in B
$. As $y_{\emptyset}\in \overline{B\setminus
B(y_{\emptyset},\delta)}^{w^*}$ then there is a countable set
$C_{y_{\emptyset}}=\{y_(i):i\in \natu\}\subset B\setminus
B(y_{\empty},\delta)$ such that
$y_{\emptyset}\in\overline{C_{y_{\empty}}}^{w^*}$.   Then a), b)
and c) are verified. By iterating this process we construct the
tree $\{y_A\}_{A\in \natu^{<\omega}}$ satisfying a), b) and c).

Now we define a new tree $\{x_A\}_{A\in\natu^{<\omega}}$ by
$x_{\emptyset}=y_{\emptyset}$ and $x_{(A,i)}=y_{(A,i)}-y_A$ for
every $i\in \natu$ and $A\in \natu^{<\omega}$. From b) we get that
$\{x_A\}_{A\in\natu^{<\omega}}$ is a seminormalized tree, since
$B$ is bounded. From c), we deduce that
$\{x_A\}_{A\in\natu^{<\omega}}$ is  topologically weak-star null.
Furthermore, if $A\in \natu^{<\omega}$ then $\sum_{C\leq
A}x_C=y_A$, from the definition of the tree
$\{x_A\}_{A\in\natu^{<\omega}}$. So
$\{x_A\}_{A\in\natu^{<\omega}}$ is a uniformly type P tree, since
$B$ is bounded and $y_A\in B$ for every $A\in \natu^{<\omega}$.
This finishes the proof of i)$\Rightarrow$ii).

ii)$\Rightarrow$i) Let $\{x_A\}$ be a seminormalized topologically
weak-star null tree such that $B=\{\sum_{C\leq
A}x_C:A\in\natu^{<\omega}\}\subset K$ and let $\delta>0$ such that
$\Vert x_A\Vert>\delta$ for every $A\in \natu^{<\omega}$. For
every $A\in \natu^{<\omega}$ and for every $n\in \natu$ we have
that $\sum_{C\leq (A,n)}x_C=\sum_{C\leq A}x_C+x_{(A,n)}$, but
$0\in\overline{\{x_{(A,n)}N\in \natu\}}^{w^*}$, since the tree
$\{x_A\}$ is topologically weak-star null. So $\sum_{C\leq
A}x_C\in\overline{\{\sum_{C\leq (A,n)}x_C:n\in\natu\}}^{w^*}$ and
$\Vert\sum_{C\leq (A,n)}x_C-\sum_{C\leq A}x_C\Vert>\delta$. This
proves that $B$ has no points where the identity map is continuous
from the weak-star to the norm topologies. In fact, we have proved
that every relatively weak-star open subset of $B$ has diameter
grater than $\delta$. Now, $\overline{B}^{\Vert\cdot\Vert}$ is a
closed and bounded subset of $K$ such that every relatively
weak-star open subset of $\overline{B}^{\Vert\cdot\Vert}$ has
diameter grater than $\delta$, and so $K$ fails $w^*$-PCP. As $K$
is $w^*$-compact, then $K$ fails RNP.\end{proof}

Essentially, the fact that RNP is separably determined has allowed
us to get the above result in the setting of general dual spaces.
The next theorem characterizes the $w^*$-PCP for subsets of dual
spaces with a separable predual in terms of $w^*$-null trees,
since in this case the $w^*$ topology is metrizable on bounded
sets. It seems natural then thinking that a characterization of
$w^*$-PCP for subsets in general dual spaces in terms of
topologically $w^*$-null trees has to be true, however we don't
know if $w^*$-PCP is separable determined in general. This is the
difference between the above proposition  and the next one, which
is obtained now easily.

\begin{proposition}\label{p1} Let $X$ be a separable Banach space and let $K$ be a closed and bounded subset  of $X^*$.
 Then the following
assertions are equivalent:\begin{enumerate} \item[i)] $K$ fails
$w^*$-PCP. \item[ii)] There is a seminormalized weak-star null
tree $\{x_A\}_{A\in\natu^{<\omega}}$ in $X^*$ such that
$\{\sum_{C\leq A}x_C:A\in\natu^{<\omega}\}\subset
K$.\end{enumerate}\end{proposition}

\begin{proof}
i)$\Rightarrow$ii) If $K$ fails $w^*$-PCP there is $B$ a subset of
$K$ and $\delta>0$ such that every relatively weak-star open
subset of $B$ has diameter greater than $2\delta$. So $b\in
\overline{B\setminus B(b,\delta)}^{w^*}$ for every $b\in B$, where
$B(b,\delta)$ stands for the open ball with center $b$ and radius
$\delta$. Note then that, since $X$ is separable the
$w^*$-topology in $X^*$ is metrizable on bounded sets, and so for
every $b\in B$ there is a countable subset $C_b\in B\setminus
B(b,\delta)$ such that $b\in \overline{C_b}^{w^*}$. Hence we can
assume that $C_b$ is a sequence $w^*$ converging to $b$. Now we
can construct, exactly like in the proof of i)$\Rightarrow$ii) of
the above proposition, the desired $w^*$-null tree satisfying ii).

ii)$\Rightarrow$i) If one assumes ii) we can repeat the proof of
  ii)$\Rightarrow$ i) in the above proposition  to get that $K$ fails
  $w^*$-PCP.\end{proof}

\begin{remark}\label{remar} If $X$ is a separable subspace of a dual space $Y^*$ with
$X$ satisfying the $w^*$-PCP, it is shown in \cite{R2} (see (1)
implies (8) of theorem 2.4 joint the comments in page 276) that
there is a separable subspace $Z$ of $Y$ such that $X$ is
isometric to a subspace of $Z^*$ and $X$ has $w^*$-PCP, as
subspace of $Z^*$. Then, in order to study the $w^*$-PCP of a
subspace of $Y^*$, it is more natural assume that $Y$ is
separable.\end{remark}

We show now our characterization of $w^*$-PCP in terms of
boundedly complete basic sequences in a general setting. A similar
characterization for PCP can be seen in \cite{LS}, but the proof
of the following result uses strongly the concept of
$w^*$-continuous boundedly complete skipped blocking finite
dimensional decomposition and assumes separability in the predual
space.

\begin{theorem}\label{p2} Let $X$, $Y$ be  Banach spaces with $Y$ separable and $X$ a subspace of $Y^*$.
 Then the following
assertions are equivalent:\begin{enumerate} \item[i)] $X$ has
$w^*$-PCP. \item[ii)] Every  weak-star null tree in $S_{X}$ is not
uniformly type P. \item[iii)] Every   weak-star null tree in
$S_{X}$ has not type P branches. \item[iv)] Every weak-star null
tree in $S_{X}$ has a boundedly complete
branch.\end{enumerate}\end{theorem}

We need the following easy

\begin{lemma}\label{le} Let $X$, $Y$ be  Banach spaces with $X$ a subspace of $Y^*$,
 and let $M$ be a finite codimensional subspace of $X$.
Assume that  $\varepsilon>0$ and  $\{x_n^*\}$ is a sequence in $X$
such that $0\in \overline{\{x_n:n\in \natu\}}^{w^*}$. If
$P:X\rightarrow N$ is a linear and relatively $w^*$-continuous
projection onto some finite dimensional subspace $N$ of $X$ with
kernel $M$ then there is $n_0\in \natu$ such that
$dist(x_{n_0}^*,M)<\varepsilon$
\end{lemma}

\begin{proof}  From
$0\in\overline{\{x_n^*:n\in \natu\}}^{w^*}$ we deduce that
$0\in\overline{\{P(x_n^*):n\in \natu\}}^{\Vert \cdot\Vert}$, since
$N$ is a finite dimensional subspace of $X$. Now, pick $n_0\in
\natu :\Vert P(x_{n_0}^*)\Vert <\varepsilon$. Then
$$dist(x_{n_0}^*,M)=\Vert x_{n_0}^*+M\Vert=\Vert P(x_{n_0}^*)+M\Vert\leq
\Vert P(x_{n_0}^*)\Vert <\varepsilon .$$
\end{proof}

{\it Proof of theorem} \ref{p2}. iv)$\Rightarrow$iii) is a
consequence of the fact that every boundedly complete basic
sequence is not type P, commented in the introduction and
iii)$\Rightarrow$ii) is trivial.

For ii)$\Rightarrow$i) it is enough applying the theorem \ref{p1}
for $K=B_{X}$ by assuming that $X$ fails  $w^*$-PCP and
normalizing.

i)$\Rightarrow$iv) Assume that $X$ has $w^*$-PCP and pick a
weak-star null tree $\{x_A\}$ in $S_{X}$.

  From \cite{R2} (see (b) of theorem 3.10 joint to the equivalence between (1) and (3) of corollary
  2.6)
  we know that every separable subspace of $Y^*$ with $w^*$-PCP
has a $w^*$-continuous boundedly complete skipped blocking finite
dimensional decomposition. As the subspace generated by the tree
$\{x_A\}$ is separable we can assume that $X$ has it , that is,
  there is a sequence $\{F_j\}$ of finite dimensional subspaces
in $X$ such that:\begin{enumerate} \item $X=[F_j:j\in
\natu]$.\item $F_k\cap[F_j:j\neq k]=\{0\}$ for every $k\in
\natu$.\item For every sequence $\{n_j\}$ of non-negative integers
with $n_j+1<n_{j+1}$ for all $j\in \natu$ and for every $f\in
[F_{(n_j,n_{j+1})}:j\in \natu]$ there exists a unique sequence
$\{f_j\}$ with $f_j\in F_{(n_j,n_{j+1})}$ for all $j\in \natu$
such that $f=\sum_{j=1}^{\infty}f_j$.\item Whenever $f_j\in
F_{(n_j,n_{j+1})}$ for all $j\in \natu$ and $\sup_n\Vert
\sum_{j=1}^{n}f_j\Vert<\infty$ then $\sum_{j=1}^{\infty}f_j$
converges.\item $F_i\cap \overline{[F_j:j\neq i]}^{w^*}= \{0\}$
for every $i$.\end{enumerate} Let $K$ be a positive constant
 such that
every skipped block sequence $\{x_j\}$ of $\{F_n\}$ with $x_j\neq
0$ for every $j$ is a boundedly complete basic sequence with
constant at most $K$.

Observe that for every $n\in \natu$ there is a linear onto
projection\break $\widetilde{P_n}:\overline{[F_i:i\geq
n]}^{w^*}\oplus[F_i:i< n]\rightarrow [F_i:i< n]$ with kernel
$\overline{[F_i:i\geq n]}^{w^*}$ and so $\widetilde{P_n}$ is $w^*$
continuous, since $\overline{[F_i:i\geq n]}^{w^*}\oplus[F_i:i< n]$
is $w^*$-closed subspace of $Y^*$ and hence a dual Banach space
and the closed graph theorem applies to $P_n$ because its kernel
is $w^*$-closed and its range is finite-dimensional. Then the
restriction of $\widetilde{P}_n$ to $X$, let us say $P_n$, is a
linear and relatively $w^*$-continuous projection from $X$ onto
$[F_i:i< n]$ with kernel $[F_i:i\geq n]$.

  We have
to construct a boundedly complete branch of the tree $\{x_A\}$.
For this, fix a sequence $\{\varepsilon_j\}$ of positive real
numbers with $\sum_{j=0}^{\infty}\varepsilon_j<1/2K$, where $K$ is
the constant of the decomposition $\{F_j\}$. Now we
 construct   a sequence $\{f_j\}$ in $X$
with $f_j\in F_{(n_j,n_{j+1})}$ for all $j$, for some increasing
sequence of integers numbers $\{n_j\}$ and a branch $\{x_{A_j}\}$
of the tree such that $\Vert x_{A_j}-f_j\Vert <\varepsilon_j$ for
all $j$. Put $n_0=0$. Then there exists $n_1>2$ and $f_0\in
F_{(n_0,n_1)}$ such that $\Vert x_{A_0}-f_0\Vert<\varepsilon_0$,
where $A_0=\emptyset$. Now, assume that $n_1,\ldots, n_{j+1}$,
$f_1,\ldots, f_j$ and $A_1,\ldots, A_j$ have been constructed. Put
$A_k=(p_1,p_2,\ldots, p_k)$ for all $1\leq k\leq j$. As the tree
is  $w^*$ null we have that
$0\in\overline{\{x_{(A_j,p)}:p\in\natu\}}^{w^*}$. Then, by the
lemma \ref{le}, we deduce that there is some $p_{j+1}\in \natu$
such that
$dist(x_{(A_j,p_{j+1})},[F_{[n_{j+1}+1,\infty)}])<\varepsilon_{j+1}$
  since $[F_{[n_{j+1}+1,\infty)}]$ is a finite
codimensional  subspace in $X$ and $P_{n_{j+1}+1}$ is relatively
$w^*$-continuous. Then there exist $n_{j+2}>n_{j+1}+1$ and
$f_{j+1}\in F_{(n_{j+1},n_{j+2})}$ such that $\Vert
x_{A_{j+1}}-f_{j+1}\Vert<\varepsilon_{j+1}$, where
$A_{j+1}=(A_j,p_{j+1})$. This finishes the inductive construction
of the branch $\{x_{A_j}\}$ satisfying that $\Vert
x_{A_j}-f_j\Vert <\varepsilon_j$ for all $j$. Finally we get that
$\sum_{j=1}^{\infty}\Vert x_{A_j}-f_j\Vert<1/2K$. Then
$\{x_{A_j}\}$ is a branch of the tree $\{x_A\}_{A\in
\natu^{<\omega}}$ which is a basic sequence equivalent to
$\{f_j\}$, being $\{f_j\}$ a skipped block sequence of $\{F_n\}$,
hence $\{x_{A_j}\}$ is a boundedly complete sequence and the proof
of theorem \ref{p2} is finished.

\bigskip

Now we can get a characterization of RNP for dual spaces,
following the above proof.

\begin{theorem}\label{r2} Let $X$ be a Banach space. Then the following
assertions are equivalent:\begin{enumerate}\item[i)] $X^*$ has
 RNP. \item[ii)] Every  topologically weak-star null tree in $S_{X}$ is not
uniformly type P. \item[iii)] Every   topologically weak-star null
tree in $S_{X}$ has not type P branches. \item[iv)] Every
topologically weak-star null tree in $S_{X}$ has a boundedly
complete branch.\end{enumerate}\end{theorem}

\begin{proof} iv)$\Rightarrow$iii) is a
consequence of the fact that every boundedly complete basic
sequence is not type P, commented in the introduction and
iii)$\Rightarrow$ii) is trivial.

For ii)$\Rightarrow$i) it is enough applying the theorem \ref{r1}
for $K=B_{X^*}$ by assuming that $X^*$ fails   RNP and
normalizing.

i)$\Rightarrow$iv) Assume that $X^*$ has RNP and pick a
topologically weak-star null tree $\{x_A\}$ in $S_{X^*}$. Call $Y$
the closed linear span of the tree $\{x_A\}$. Now $Y$  is a
separable subspace of $X^*$ and then there is a separable subspace
$Z$ of $X$ norming $Y$ so that $Y$ is isometric to a subspace of
$Z^*$. As $X^*$ has RNP, we get that $Z^*$ has RNP. Hence $Y$ is a
separable subspace of $Z^*$, being $Z$ a separable space, and
$Z^*$ has $w^*$-PCP since $Z^*$ has RNP. Observe that the tree
$\{x_A\}$ is now a topologically weak-star null tree in $S_Z^*$,
so we can select a full weak-star null subtree of $\{y_A\}$, since
$Z$ is separable and so the $w^*$ topology in $Z^*$ is metrizable
for bounded sets. We apply the proof of i) $\Rightarrow$ iv) in
the above result with $X=Y=Z^*$ to get a boundedly complete branch
of $\{y_A\}$. As $\{y_A\}$ is a full subtree of $\{x_A\}$, the
branches of $\{y_A\}$ are branches of $\{x_A\}$ and $\{x_A\}$ has
a boundedly complete branch.\end{proof}

In the case $X$ is a separable Banach space, the above result can
be written in a terms of weak-star null trees. Then we get as an
immediate consequence in the  following corollary  a result
obtained in \cite{DF} in a different way.

\begin{corollary} Let $X$ be a separable Banach space. Then $X^*$
is separable  (equivalently, $X^*$ has RNP) if, and only if, every
weak-star null tree in $S_{X^*}$ has a boundedly complete
branch.\end{corollary}

\begin{proof} When $X$ is separable, the weak-star topology in
$X^*$ is metrizable on bounded sets and so every topologically
weak-star null tree in $S_{X^*}$ has a full subtree which is
weak-star null. With this in mind, it is enough to apply the above
theorem to conclude, since $X^*$ is separable if, and only if,
$X^*$ has RNP, whenever $X$ is separable. \end{proof}

The following consequence, obtained in a different way  in
\cite{R1}, shows how many separable and dual subspaces contains
every Banach space with PCP.

\begin{corollary} Let $X$ a Banach space with PCP. Then every
seminormalized basic sequence in $X$ has a boundedly complete
subsequence.\end{corollary}

\begin{proof} Pick $\{x_n\}$ a seminormalized basic sequence in
$X$.  Then either $\{x_n\}$ has a subsequence equivalent to the
unit vector basis of $\ell_1$, and hence boundedly complete, or
$\{x_n\}$ has a weakly Cauchy subsequence which we denotes again
$\{x_n\}$.

In the case $\{x_n\}$ is weakly convergent we get that $\{x_n\}$
is weakly null, because it is a basic sequence. Now, $\{x_n\}$ is
a seminormalized weakly null tree in $X$ and hence $\{x_n\}$ is a
seminormalized weak-star null tree in $X^{**}$. As $[x_n]$ is a
separable subspace of $X^{**}$ with $w^*$-PCP, from \ref{remar}
there is a separable subspace $Z\subset X^*$ such that $[x_n]$ is
an isometric subspace of $Z^*$ with $w^*$-PCP. Therefore $\{x_n\}$
is a seminormalized weak-star null tree in $[x_n]$, which is a
subspace of $Z^*$ with $w^*$-PCP, being $Z$ separable. From
theorem, we get a boundedly complete branch and so a boundedly
complete subsequence.

If $\{x_n\}$ is not weakly convergent we can apply the
$c_0$-theorem \cite{R3} to get a strongly summing subsequence,
denoted again by $\{x_n\}$, since $X$ has PCP and so $X$ does not
contain $c_0$. Let $x^{**}=w^*-lim_n\ x_n\in X^{**}$. Now
$\{x_n-x^{**}\}$ is a weak-star null sequence in $X\oplus
[x^{**}]>\subset X^{**}$. As $X$ has PCP, we get that $X$ has
$w^*$-PCP as a subspace of $X^{**}$, then it is easy to see that
$X\oplus [x^{**}]\subset X^{**}$ has $w^*$-PCP. Now $[x_n-x^{**}]$
is a separable subspace of $X^{**}$ with $w^*$-PCP and then, from
\ref{remar} there is $Z$ a separable subspace of $X^*$ such that
$[x_n-x^{**}]$ is an isometric subspace of $Z^*$ with the
$w^*$-PCP, being $Z$ separable. From theorem, we get a boundedly
complete branch and so a boundedly complete subsequence, denoted
again by $\{x_n-x^{**}\}$. So we have that $\{x_n-x^{**}\}$ is
boundedly complete and $\{x_n\}$ is strongly summing. Let us see
that $\{x_n\}$ is boundedly complete. Indeed, if for some sequence
of scalars $\{\lambda_n\}$ we have that
$\sup_n\Vert\sum_{k=1}^n\lambda_n x_n\Vert<\infty$, then the
series $\sum_n\lambda_n$ is convergent, since $\{x_n\}$ is
strongly summing. Now it is clear that
$\sup_n\Vert\sum_{k=1}^n\lambda_n (x_n-x^{**})\Vert<\infty$ and
then $\sum_n\lambda_n x_n-x^{**}$ converges, since
$\{x_n-x^{**}\}$ is boundedly complete. So $\sum_n\lambda_n x_n$
converges, since $\sum_n\lambda_n$ is convergent, and $\{x_n\}$ is
boundedly complete.\end{proof}

The converse of the above result is false, even for Banach spaces
not containing $\ell_1$ (see \cite{GL}).

Now we pass to show some consequences about the problem of the
determination of $w^*$-PCP by subspaces with a basis. We begin by
proving that every seminormalized weak-star null tree has a basic
full $w^*$-null subtree. The same result is then true for the weak
topology, by considering $X$ as a subspace of $X^{**}$. We don't
know exact reference for this result, so we give a proof based on
the Mazur proof of the known result that every seminormalized
sequence in a dual Banach space such that  $0$ belongs to  its
weak-star closure has a basic subsequence.

\begin{lemma}\label{lema3} Let $X$ be a Banach space and
$\{x_A\}_{A\in\natu^{<\omega}}$ a seminormalized  $w^*$ null tree
in $S_{X^*}$. Then for every $\varepsilon>0$ there is a full basic
subtree still $w^*$ null with basic constant less than
$1+\varepsilon$\end{lemma}

\begin{proof} Let $\phi:\natu^{<\omega}\rightarrow \natu\cup\{0\}$
be a fixed bijective map such that $\phi(\emptyset)=0$,
$\phi(A)\leq\phi(B)$ whenever $A\leq B\in\natu^{<\omega}$ and
$\phi(A,i)\leq\phi(A,j)$ whenever $A\in\natu^{<\omega}$ and $i\leq
j$. Fix also $\varepsilon>0$ and a sequence of positive numbers
$\{\varepsilon_n\}_{n\geq 0}\in (0,1)$ such that
$\frac{1+\sum_{n=0}^{\infty}\varepsilon_n}{\prod_{n=0}^{\infty}(1-\varepsilon_n)}
<1+\varepsilon$. Now we proceed by induction to construct the
desired subtree $\{y_A\}_{A\in\natu^{<\omega}}$, following the
order given by $\phi$  to define for every $A\in\natu^{<\omega}$
$y_A$ and   get in this way the full condition. That is, we have
to prove that for every $n\in\natu\cup \{0\}$ we can construct
$y_{\phi^{-1}(n)}$  such that $\{y_{A}\}_{A\in\natu^{<\omega}}$ is
a  $w^*$ null full subtree satisfying that for every
$n\in\natu\cup\{0\}$ there is a finite set
$\{f_1^n,\ldots,f_{k_n}^n\}\subset S_{X}$ such
that\begin{enumerate}\item[i)] $\{f_1^n,\ldots,f_{k_n}^n\}$ is a
$(1-\varepsilon_n)$-norming  set for
$Y_n=[y_{\phi^{-1}(0)},\ldots, y_{\phi^{-1}(n)}]$. \item[ii)]
$\vert f_i^n(y_{\phi^{-1}(n+1)})\vert <\varepsilon_n$ for every
$i$.\item[iii)] For every $A\in \natu^{<\omega}$ there is an
increasing map $\sigma_A:\natu\rightarrow\natu$ such that
$y_{(A,i)}=x_{(A,\sigma_{A}(i))}$ for every $i$.\end{enumerate}

For $n=0$, we have $\phi^{-1}(0)=\emptyset$ and then we define
$y_{\emptyset}=x_{\emptyset}$. Now take $f_1^0\in S_{X}$
$(1-\varepsilon_0)$-norming the subspace $Y_0=[y_{\emptyset}]$. As
the tree $\{x_A\}_{A\in\natu^{<\omega}}$ is $w^*$ null there is
$p_0\in\natu$ such that $\vert f_1^0(x_{(p)})\vert <
\varepsilon_0$ for every $p\geq p_0$.   Then we do
$y_{\phi^{-1}(1)}=x_{(p_0)}$. As $\phi(A,i)\leq\phi(A,j)$ whenever
$A\in\natu^{<\omega}$ and $i\leq j$ we deduce that
$\phi^{-1}(1)=(1)$ and  define $\sigma_{\emptyset}(1)=p_0$ so that
 $y_{(\emptyset,1)}=x_{(\emptyset ,\sigma_{\emptyset}(1))}$.

Assume $n\in\natu$ and that we have already defined
$y_{\phi^{-1}(0)},\ldots ,y_{\phi^{-1}(n-1)}$. Now
$\phi^{-1}(n)-<\phi^{-1}(n)$, then $\phi(\phi^{-1}(n)-)<n$ and so
$y_{\phi^{-1}(n)-}$ has been already constructed, by induction
hypotheses. Put $\phi^{-1}(n)=(A,h)$ for some $h\in\natu$, where
$A=\phi^{-1}(n)-$. As $\phi(A,k)\leq \phi(A,h)$ for $k\leq h$ we
have that $y_{(A,k)}$ has been constructed with
$y_{(A,k)}=x_{(A,\sigma_{A}(k))}$ whenever $k< h$ and
$\sigma_{A}(k)$ has been constructed strictly increasing for $k<
h$. Put $Y_{n-1}=[y_{\phi^{-1}(0)},\ldots ,y_{\phi^{-1}(n-1)}]$
and pick $f_1^{n-1},\ldots ,f_{k_{n-1}}^{n-1}$ elements in
$S_{Y_{n-1}}$ $(1-\varepsilon_{n-1})$-norming $Y_{n-1}$. As the
tree $\{x_A\}_{A\in\natu^{<\omega}}$ is $w^*$ null there is
$p_{n-1}>\max_{k<h}\sigma_{A}(k)$ such that  $  \vert
f_i^{n-1}(x_{(A,p)})\vert <\varepsilon_{n-1},\ 1\leq i\leq
k_{n-1}$ for every $p\geq p_{n-1}$. Then we do
$y_{\phi^{-1}(n)}=x_{(A,p_{n-1})}$ and $\sigma_{A}(h)=p_{n-1}$.
Then $\sigma_{A}(k)$ is constructed being strictly increasing for
$k\leq h$ and $y_{\phi^{-1}(n)}=y_{(A,h)}=x_{(A,\sigma_A(h))}$.
Now the construction of the subtree $\{y_A\}$ is complete
satisfying i), ii) and iii). From the construction we get that
$\{y_A\}$ is a full and $w^*$-null subtree.

Let us see that $\{y_{\phi^{-1}(n)}\}$ is a basic sequence in $X$.
Put $z_n=y_{\phi^{-1}(n)}$, fix $p<q\in\natu$ and compute
$\Vert\sum_{i=1}^q\lambda_iz_i\Vert$, where $\{\lambda_i\}$ is a
scalar sequence. Assume that
$\Vert\sum_{i=1}^q\lambda_iz_i\Vert\leq 1$.   From i) pick $j$
such that $\vert f_j^{q-1}(\sum_{i=1}^{q-1}\lambda_iz_i)\vert
>(1-\varepsilon_{q-1})\Vert\sum_{i=1}^{q-1}\lambda_iz_i\Vert.$ Then we have  from ii)
$\vert f_j^{q-1}(z_q)\vert <\varepsilon_{q-1}$ and so

$$\Vert\sum_{i=1}^q\lambda_iz_i\Vert\geq \vert
f_j^{q-1}(\sum_{i=1}^q\lambda_iz_i)\vert
>(1-\varepsilon_{q-1})\Vert\sum_{i=1}^{q-1}\lambda_iz_i\Vert -\varepsilon_{q-1}.$$

By repeating this computation we get
$$\Vert \sum_{i=1}^q\lambda_iz_i\Vert\geq
(\prod_{i=p+1}^q(1-\varepsilon_{i-1}))\Vert
\sum_{i=1}^p\lambda_iz_i\Vert
-(\sum_{i=p+1}^q\varepsilon_{i-1}),$$ and so,
$$\Vert\sum_{i=1}^p\lambda_iz_i\Vert\leq
\frac{1+\sum_{i=p+1}^q\varepsilon_{i-1}}{\prod_{i=p+1}^q(1-\varepsilon_{i-1})}<1+\varepsilon$$
The last inequality proves that $\{z_n\}$ is a basic sequence in
$X$ with basic constant less than $1+\varepsilon$ and the proof is
complete.\end{proof}

We don¬t know if the above result is still true changing weak-star
 null by topologically weak-star null.

The following result shows that $w^*$-PCP is determined by
subspaces with a Schauder basis in the natural setting of dual
spaces of separable Banach spaces.

\begin{corollary}\label{fin} Let $X$, $Y$ be Banach spaces such that $Y$ is
separable and $X$ is a subspace of $Y^*$. Then $X$ has $w^*$-PCP
if, and only if, every subspace of $X$ with a Schauder basis has
$w^*$-PCP.\end{corollary}

\begin{proof} Assume that $X$ fails $w^*$-PCP. Then, from theorem \ref{p2},
there is a $w^*$-null tree in the unit sphere of $X$ without
boundedly complete branches. Now, from lemma \ref{lema3}, we can
extract a $w^*$-null full basic subtree. The subspace generated by
this subtree $Z$ is a subspace of $X$ with a Schauder basis
containing a $w^*$-null tree in $S_X$ without boundedly complete
branches, from the full condition, so $Z$ fails $w^*$-PCP, from
theorem \ref{p2}\end{proof}

As a consequence we get, for example, that a subspace of
$\ell_{\infty}$, the space of bounded scalar sequences with the
sup norm, failing the $w^*$-PCP ( or failing PCP) contains a
further subspace with a Schauder basis failing the $w^*$-PCP.

If we do $X=Y^*$ in the above corollary one deduces the following

\begin{corollary} Let $X$ be a separable Banach space. Then $X^*$ has
RNP if, and only if, every subspace of $X^*$ with a Schauder basis
has $w^*$-PCP.\end{corollary}


\bigskip






\begin{thebibliography}{999999}


\bibitem{B} J. Bourgain. {\it Dentability and finite-dimensional
decompositions.} Studia Math. 67 (1980), 135-148.




\bibitem{bou} R. D. Bourgin. {\it Geometric Aspects of Convex Sets
with the Radon-Nikod{\'y}m Property.} Lecture Notes in Math. 993.
Springer-Verlag, Berlin 1983.


\bibitem{DF} S. Dutta, V. P. Fonf. {\it On tree characterizations of $G_{\delta}$-embeddings and some Banach
spaces.} Israel J. Math. 167 (2008), 27-48.

\bibitem{GM}
N. Ghoussoub, B. Maurey. $G_{\delta}-${\it embeddings in Hilbert
spaces}. J. Funct. Anal. 61 (1985), 72-97.

\bibitem{GL} G. L{\'o}pez-P{\'e}rez. {\it Banach spaces with many boundedly complete basic sequences failing
PCP}. J. Funct. Anal. 259 (2010), 2139-2146.


\bibitem{LS} G. L{\'o}pez-P{\'e}rez, J. Soler-Arias. {\it A tree
characterization of the point of continuity property in general
Banach spaces.} Proc. Amer. MAth. Soc. 140 (2012), 4243-4245.



\bibitem{LZ} J. Lindenstrauss, L. Tzafriri. {\it Classical Banach
Spaces I.} Springer Verlag. Berlin 1977.



\bibitem{R1} H. P. Rosenthal. {\it Boundedly complete weak-Cauchy basic
sequences in Banach spaces with PCP}. J. Funct. Anal. 253 (2007),
772-781.

\bibitem{R2} H. P. Rosenthal. {\it Weak$^*$-Polish Banach Spaces}.
J. Func. Anal. 76 (1988), 267-316.

\bibitem{R3} H. P. Rosenthal. {\it A subsequence principle characterizing
Banach spaces containing $c_{0}$.} Bulletin of the Amer. Math.
Soc. 30 (2) (1994), 227-233.

\bibitem{S} I. Singer.  {\it Bases in Banach spaces I.}
Springer-Verlag. Berlin 1970.





\end{thebibliography}
\end{document}